\documentclass{article}
\usepackage[utf8]{inputenc}

\usepackage{cite}

\usepackage{amssymb, amsmath, amsthm}
\usepackage[margin=1in]{geometry}
\usepackage{tikz}
\usepackage{stmaryrd} 
\usepackage{thm-restate}
\usepackage{xcolor}
\usepackage{graphicx}
\usepackage{etoolbox}

\usepackage{hyperref}
    \hypersetup{colorlinks,
      linkcolor=blue,
      citecolor=blue,
      urlcolor=blue}

\def\ucr#1{{\text{\rm cr}(#1)}}
\def\rcr#1{\overline{{\text{\rm cr}}}(#1)}
\def\Hom{{\text{\rm Hom}^+(\mathcal{A},\mathbb{R})}}

\tikzset{flag_pic/.style={scale=1}}  
\tikzset{unlabeled_vertex/.style={inner sep=1.7pt, outer sep=0pt, circle, fill}} 
\tikzset{labeled_vertex/.style={inner sep=2.2pt, outer sep=0pt, rectangle, fill=yellow, draw=black}} 
\tikzset{edge_color0/.style={color=black,line width=1.2pt,opacity=0.5}} 
\tikzset{edge_color1/.style={color=red,  line width=1.2pt,opacity=0}} 
\tikzset{edge_color2/.style={color=blue, line width=1.2pt,opacity=1}} 
\tikzset{edge_color3/.style={color=green!80!black,line width=1.2pt}} 
\tikzset{edge_color4/.style={color=red,  line width=1.2pt,dotted}} 
\tikzset{edge_color5/.style={color=blue, line width=1.2pt,dotted}} 
\tikzset{edge_color6/.style={color=green, line width=1.2pt,dotted}} 
\tikzset{edge_color7/.style={color=orange, line width=1.2pt}} 
\tikzset{edge_color8/.style={color=gray, line width=1.2pt}} 
\tikzset{edge_colorroot/.style={color=red, line width=1.7pt}} 
\tikzset{edge_thin/.style={color=black}} 
\tikzset{edge_hidden/.style={color=black,dotted,opacity=0}} 
\tikzset{vertex_color1/.style={inner sep=1.7pt, outer sep=0pt, draw, circle, fill=red}} 
\tikzset{vertex_color2/.style={inner sep=1.7pt, outer sep=0pt, draw, circle, fill=blue}} 
\tikzset{vertex_color3/.style={inner sep=1.7pt, outer sep=0pt, draw, circle, fill=green}} 
\tikzset{vertex_color4/.style={inner sep=1.7pt, outer sep=0pt, draw, circle, fill=pink}} 
\tikzset{labeled_vertex_color1/.style={inner sep=2.2pt, outer sep=0pt, draw, rectangle, fill=red}} 
\tikzset{labeled_vertex_color2/.style={inner sep=2.2pt, outer sep=0pt, draw, rectangle, fill=blue}} 
\tikzset{labeled_vertex_color3/.style={inner sep=2.2pt, outer sep=0pt, draw, rectangle, fill=green}} 
\tikzset{labeled_vertex_color4/.style={inner sep=2.2pt, outer sep=0pt, draw, rectangle, fill=pink}} 

\tikzset{vertexA/.style={inner sep=1.7pt, outer sep=0pt, circle, draw=black, fill=red}} 
\tikzset{vertexB/.style={inner sep=1.7pt, outer sep=0pt, circle, draw=black, fill=blue}}

\newtheorem{theorem}{Theorem}

\newtheorem{lemma}[theorem]{Lemma}

\newtheorem{conjecture}[theorem]{Conjecture}

\begin{document}
\title{Crossing numbers of complete bipartite graphs}

\author{J\'ozsef Balogh\thanks{Department of Mathematical Sciences, University of Illinois, Urbana, IL 61801, USA
E-mail: {\tt jobal@illinois.edu}.
Balogh was supported in part by NSF grants DMS-1764123, RTG DMS-1937241 and FRG DMS-2152488, the Arnold O. Beckman Research Award (UIUC Campus Research Board RB 22000), and the Langan Scholar Fund (UIUC).}
\and
Nicholas Layman\thanks{Department of Mathematics, Iowa State University, Ames, IA 50011, USA 
E-mail: {\tt nlayman@iastate.edu}. 
}
\and
Bernard Lidick\'{y}\thanks{Department of Mathematics, Iowa State University, Ames, IA. E-mail: {\tt lidicky@iastate.edu}. Research of this author is supported in part by NSF FRG DMS-2152490 and DMS-2554129, the Simons Foundation TSM-00013439 and Scott Hanna Professorship.
}
\and
Florian Pfender\thanks{Department of Mathematical and Statistical Sciences, University of Colorado Denver, CO 80204 USA
E-mail: {\tt Florian.Pfender@ucdenver.edu}.
Research is partially supported by NSF DMS-2152498 and  DMS-2554130.
}
\and
Gelasio~Salazar\thanks{Instituto de Física, Universidad Autónoma de San Luis Potosí, SLP 78000, Mexico
E-mail: {\tt gsalazar@ifisica.uaslp.mx}
Supported by CONACYT under Proyecto Ciencia de Frontera 191952.
}
\and
Sam Spiro\thanks{Georgia State University, Atlanta, GA, USA, 
E-mail: {\tt sspiro@gsu.edu}.
Research supported by the NSF Mathematical Sciences Postdoctoral Research Fellowships Program under Grant DMS-2202730.}
\and
Jan Volec\thanks{Department of Theoretical Computer Science, Faculty of Information Technology, Czech Technical University in Prague, Th\'akurova 9, Prague, 160 00, Czech Republic. E-mail: {\tt jan@ucw.cz}
Research of this author was supported by the grant
23-06815M of the Grant Agency of the Czech Republic.
}
}

\maketitle

\def\crKnnCertified{0.96936}  
\def\crKnnLemmaCertified{0.24234}

\begin{abstract}
Zarankiewicz's long standing  conjecture states that the crossing number  $\ucr{K_{m,n}}$ of the complete bipartite graph is
\[
Z(m,n):=\bigg\lfloor{\frac{m}{2}}\bigg\rfloor \bigg\lfloor{\frac{m-1}{2}}\bigg\rfloor \bigg\lfloor{\frac{n}{2}}\bigg\rfloor \bigg\lfloor{\frac{n-1}{2}}\bigg\rfloor.
\]
Using flag algebras, we show that $\ucr{K_{n,n}} \geq \crKnnCertified \cdot Z(n,n) + o(n^4)$.

We also show that the rectilinear crossing number
$\overline{cr}(K_{n,n})$ of $K_{n,n}$ is at least $0.987\cdot Z(n,n) + o(n^4)$. 
Additionally, we prove the following for complete bipartite graphs:  if a drawing of $K_{n,n}$ contains no induced subdrawing of $K_{3,4}$ with exactly two crossings whose endpoints intersect in exactly one vertex, then it has at least $Z(n,n)+o(n^4)$ crossings. 
This is a local restriction inspired by Tur\'an-type problems that gives an asymptotically tight result.

A closely related conjecture is Hill's conjecture stating that the crossing number $\ucr{K_{n}}$ of the complete graph $K_n$ is
\[
H(n):=
\frac14 
\left\lfloor \frac{n}{2} \right\rfloor
\left\lfloor \frac{n-1}{2} \right\rfloor
\left\lfloor \frac{n-2}{2} \right\rfloor
\left\lfloor \frac{n-3}{2} \right\rfloor
\]
for all $n \geq 3$.
By an averaging argument, the asymptotic Zarankiewicz conjecture implies the asymptotic Hill conjecture. 
We improve  the lower bound on  $\ucr{K_{n}}$  from $(0.985+o(1)) H(n)$ to $(0.9899+o(1)) H(n)$.
Both conjectures were claimed to be true in manuscripts by OpenAI dated Sep 23, 2026 and made public on Oct 6, 2026.
This paper represents the state of knowledge before those publications and shows how far we can push the flag algebra method.
\end{abstract}

%

\section{Introduction}

We recall that the {\em crossing number} $\ucr{G}$ of a graph $G$ is the minimum number of pairwise crossings of edges in a drawing of $G$ in the plane.

The systematic study of crossing numbers of graphs seems to date back to a question by Paul Tur\'an in 1944. While working in a brick factory outside Budapest, where his work consisted of pushing wagon loads of bricks from kilns to storage sites, Tur\'an noted that rail wagons tended to derail at crossings of tracks, and he wondered what was the best way to design a brickyard, in the sense that the number of track crossings is minimized. In modern graph theory terminology, this question is equivalent to: what is the crossing number cr$(K_{m,n})$ of the complete bipartite graph $K_{m,n}$? As pointed out by Beineke and Wilson in their lively survey~\cite{Beineke2010}, this problem has a curious history. Zarankiewicz thought he had solved the problem~\cite{Zarankiewicz54}. Around a decade later, his main argument was found to have an unpluggable gap~\cite{Decline}. In any case, Zarankiewicz did exhibit a natural way to draw $K_{m,n}$ with exactly 
\[
Z(m,n):=\bigg\lfloor{\frac{m}{2}}\bigg\rfloor \bigg\lfloor{\frac{m-1}{2}}\bigg\rfloor \bigg\lfloor{\frac{n}{2}}\bigg\rfloor \bigg\lfloor{\frac{n-1}{2}}\bigg\rfloor
\]
crossings, see Figure~\ref{fig:K66} for an example. To this day, no drawing of $K_{m,n}$ with fewer than $Z(m,n)$ crossings has been found, and cr$(K_{m,n})=Z(m,n)$ has become known as {\em Zarankiewicz's Conjecture}. Before the recent OpenAI announcement on Oct 6, 2026, claiming a full proof, this conjecture had been settled only when $\min\{m,n\}\le 6$~\cite{Kleitman70}, and in a few additional small cases~\cite{Woodall93}. 

\begin{figure}
\begin{center}    
\tikzset{
insep/.style={inner sep=2pt, outer sep=0pt, circle, fill}, 
redline/.style={red,line width=2pt},
blueline/.style={blue,line width=2pt},
grayline/.style={gray,line width=2pt},
blackline/.style={gray,line width=2pt},
turanblock/.style={fill=blue!30},
funky/.style={black,line width=2pt,dotted},
full/.style={black,line width=2pt,latex},
turan/.style={inner sep=0.43cm, outer sep=0pt, circle, fill=white!30,draw}, 
trash/.style={inner sep=0.2cm, outer sep=0pt, circle, fill=white!30,draw}, 
insepred/.style={inner sep=4pt, outer sep=0pt, circle, fill=red,draw}, 
insepgreen/.style={inner sep=4pt, outer sep=0pt, circle, fill=white,draw}, 
insepblue/.style={inner sep=4pt, outer sep=0pt, circle, fill=blue,draw}, 
}

\begin{tikzpicture} 
\draw
\foreach \a in {0,1}{
\foreach \z/\n in {0.7/1, 1.3/2, 1.9/3}
{(0,0)  +(\a*90:\z) node[vertexA](x\a\n){}}}

\foreach \a in {0,1}{
\foreach \z/\n in {0.7/4,1.3/5, 1.9/6}
{(0,0)  +(\a*90+180:\z) node[vertexB](x\a\n){}}}

\foreach \x in {x01,x02,x03,x04,x05,x06}{
\foreach \y in {x11,x12,x13,x14,x15,x16}{
(\x) -- (\y)
}
}

\foreach \x in {x01,x02,x03,x04,x05,x06} { (\x) node[vertexA]{} }
\foreach \x in {x11,x12,x13,x14,x15,x16} { (\x) node[vertexB]{} }
;

\end{tikzpicture}
\end{center}
    \caption{Zarankiewicz's drawing of $K_{6,6}$.}
    \label{fig:K66}
\end{figure}
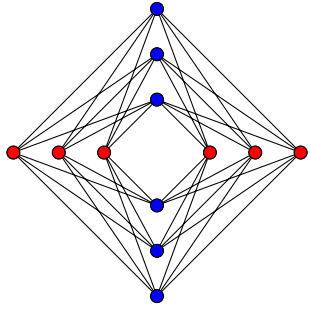

There seemed to be a consensus among researchers interested in the conjecture that the conjecture is true, but that attempting to settle it exactly is a hopeless task. 
One of the reasons may be that there are many constructions that have asymptotically $Z(m,n)$ crossings. For example, every geodesic drawing of $K_{2m,2n}$ on the sphere such that for each vertex there is another vertex in the same class drawn in its antipodal location has exactly $Z(2m,2n)$ crossings. For this reason, most recent efforts on this problem were focused on its asymptotic version. The asymptotic version of Zarankiewicz's Conjecture reads $\ucr{K_{m,n}}=m^2n^2/16\, + o(m^2 n^2)$. This weaker version also remained open. 
We focus on the diagonal case which reads as follows. 
\begin{conjecture}\label{conj:Z}
$ \ucr{K_{n,n}}=n^4/16\, + o(n^4)$.
\end{conjecture}
To compare the asymptotic results, we let
\[
L := \lim_{n \to \infty} \frac{\ucr{K_{n,n}}}{Z(n,n)},
\]
\noindent so that Conjecture~\ref{conj:Z} may be equivalently paraphrased by saying that $L=1$.

Kleitman~\cite{Kleitman70} showed that $\ucr{K_{5,n}} = Z(5, n)$, and an easy counting argument shows that  this implies that  $L \geq 0.8$. De~Klerk, Maharry, Pasechnik, Richter, and Salazar~\cite{deKlerk06} gave a lower bound on $\ucr{K_{7,n}}$ using semidefinite programming (SDP) implying $L > 0.83$. 
A further improvement to $L > 0.859$ follows from 
de Klerk, Pasechnik, and Schrijver~\cite{deKlerk07} who also used SDP to give a lower bound on $\ucr{K_{9,n}}$.
For a fixed $m \geq 3$, deciding whether or not Zarankiewicz’s conjecture holds for $K_{m,n}$ is a finite problem for every
$n \geq m$, see~\cite{deKlerk06}.
Brosch and Polak~\cite{Brosch2022} used SDP to provide lower bounds on $\ucr{K_{10,n}}$, $\ucr{K_{11,n}}$, $\ucr{K_{12,n}}$, and $\ucr{K_{13,n}}$ giving $L \geq 0.8878$.  
Very recently, Brosch and Pikhurko~\cite{Brosch2026} managed to refine the argument further using AI and obtained the following result.
\begin{theorem}[Brosch and Pikhurko~\cite{Brosch2026}]
    $L \geq 0.935$.  In other words $\ucr{K_{n,n}} > 0.935 n^4/16+o(n^4)$.
\end{theorem}

 Norin and Zwols~\cite{NorinZwols} announced in 2014 the following.
\begin{theorem}[Norin and Zwols~\cite{NorinZwols}]\label{thm:NZKnn}
$L > 0.905$. In other words $\ucr{K_{n,n}} > 0.905 n^4/16+o(n^4)$.
\end{theorem}
Their result was presented in a lecture at Banff Research Center, but it was never published in a paper. Norin and Zwols were the first to apply flag algebras to crossing numbers. Their result inspired other applications of flag algebras in graph drawing and geometry~\cite{Gethner2016,BaloghHill2019,goaoc2018}. 
By utilizing a semidefinite programming solver developed by Lidick\'y and Volec~\cite{lpsdp2026}, we were able to push the flag algebras calculation further and obtained the following improvement. 
\begin{theorem}\label{thm:Knn}
$L > \crKnnCertified$. In other words $\ucr{K_{n,n}} > \crKnnCertified n^4/16+o(n^4)$.
\end{theorem}
In a conference proceedings version of this paper~\cite{Balogh2023conf}, we show a significantly weaker lower bound $\ucr{K_{n,n}} > 0.9118 \cdot Z(n,n) + o(n^4)$.

It is common to look at restrictions of the drawings to try and get improved bounds. We recall that in a {\em rectilinear drawing} of a graph each edge is drawn as a straight segment, and the \emph{rectilinear crossing number} $\rcr{G}$ of a graph $G$ is the minimum number of pairwise crossings of edges in a rectilinear drawing of $G$. Note that Zarankiewicz's original drawings of $K_{m,n}$ are actually rectilinear. For the rectilinear version of Zarankiewicz's conjecture, we can further improve the bound from Theorem~\ref{thm:Knn}.
\begin{theorem}\label{thm:rKnn}
$\overline{cr}(K_{n,n}) > 0.987 n^4/16+o(n^4)$.
\end{theorem}

Another restriction one can make is that some given subdrawing is forbidden. 
This sort of condition is reminiscent of Tur\'an-type problems from extremal graph theory, which go as follows.
Let $F$ be a graph. What is the maximum number of edges in an $n$-vertex graph not containing $F$ as a (possibly induced) subgraph?
In our case, let $F$ be a drawing of $K_{3,4}$ with two crossings sharing exactly one vertex, see Figure~\ref{fig:K34}.
For example, for a vertex set $\{u_1,u_2,u_3,v_1,v_2,v_3,v_4\}$ the pairs of crossing edges may be $(u_1v_1, u_2v_2)$ and $(u_1v_4, u_3v_3)$.
This creates a local combinatorial constraint rather than a geometrical constraint. Observe that $F$ does not appear in Zarankiewicz's drawing of $K_{n,n}$. This restriction allows us to give an asymptotically correct bound. 

\begin{theorem}\label{thm:tKnn}
    If $K^D_{n,n}$ is a drawing of $K_{n,n}$ where no $K_{3,4}$ induces exactly two crossings sharing one vertex, then $K^D_{n,n}$ has at least $n^4/16 + o(n^4)$ crossings.
\end{theorem}

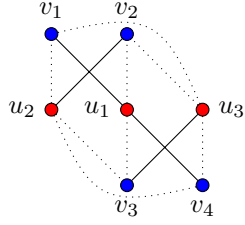
\begin{figure}
\begin{center}
    \begin{tikzpicture}
    \draw 
    (0,0) node[vertexA,label=left:{$u_2$}](u2){} 
    (1,0) node[vertexA,label=left:{$u_1$}](u1){} 
    (2,0) node[vertexA,label=right:{$u_3$}](u3){} 
    
    (0,1) node[vertexB,label=above:{$v_1$}](v1){} 
    (1,1) node[vertexB,label=above:{$v_2$}](v2){} 
    (1,-1) node[vertexB,label=below:{$v_3$}](v3){} 
    (2,-1) node[vertexB,label=below:{$v_4$}](v4){} 
    ;
    \draw
    (u1)--(v1) (u2)--(v2) (u3)--(v3) (u1)--(v4)
    ;
    \draw[dotted]
    (u1)--(v2)
    (u1)--(v3)
    (u2)--(v1)
    (u3)--(v4)
    (u2)--(v3)
    (u3)--(v2)
    (u2) to[bend right=40,looseness=1.5] (v4)
    (u3) to[bend right=40,looseness=1.5] (v1)
    ;
    \end{tikzpicture}
\end{center}
    \caption{We let $F$ denote a drawing of $K_{3,4}$ with two crossings sharing exactly one vertex, such as the one illustrated above. Here the solid edges are crossed, while the dotted edges are uncrossed.}
    \label{fig:K34}
\end{figure}

The problem of crossing number of the complete graph $K_n$ was first investigated by the artist Antony Hill in the late 1950s. 
The drawings of complete graphs found by Hill~\cite{Harary1963} are conjectured to be best possible.
\begin{conjecture}[Hill's conjecture]\label{conj:h}
\[
\ucr{K_n} = H(n):=
\frac14 
\left\lfloor \frac{n}{2} \right\rfloor
\left\lfloor \frac{n-1}{2} \right\rfloor
\left\lfloor \frac{n-2}{2} \right\rfloor
\left\lfloor \frac{n-3}{2} \right\rfloor.
\]
\end{conjecture}

Balogh, Lidick\'y and Salazar~\cite{BaloghHill2019} used flag algebras on the crossing number of the complete graph and obtained $\ucr{K_n} \geq 0.985 H(n) + o(n^4)$.
By running a larger program, we obtained an improvement.
\begin{theorem}\label{thm:hill}
$\ucr{K_n} \geq 0.9899 H(n) + o(n^4).$
\end{theorem}

In the next section we give a brief introduction to flag algebras, the main method used in this paper. Then we describe the specific setup we use in this paper. 
Sections~\ref{sec:knn}, \ref{sec:rectinilear}, \ref{sec:turan} and \ref{sec:hill} deal with the proofs of Theorems~\ref{thm:Knn},~\ref{thm:rKnn}, \ref{thm:tKnn} and \ref{thm:hill}, respectively.

\subsection{Notes on AI}
Both Conjectures~\ref{conj:Z} and \ref{conj:h} were claimed to be true by OpenAI~\cite{OAI:The-crossing-number-of-complete-bipartite-graphs-September-23-2026,OAI:The-crossing-number-of-complete-graphs-September-23-2026}.
 ChatGPT-6 was used to assist with writing Section~\ref{sec:generate}.
It was also used to check for typos and correctness of the manuscript.

\subsection{Notation}
We distinguish a graph $G$ from its drawing by denoting a drawing of $G$ by $G^D$. 
We denote the crossing number of a graph $G$ by $\ucr{G}$. For a drawing $G^D$, $\ucr{G^D}$ is the number of crossings in $G^D$. 
For a set of vertices $X$, we denote the induced subgraph of $G$ by $G[X]$ and $G^D[X]$ denotes the subdrawing of $G^D$ induced by $X$. 

\section{Flag Algebras}\label{sec:fa}

\subsection{General Bipartite Setup}
Flag algebra method is a tool developed by Razborov~\cite{Razborov2007} for solving problems in extremal combinatorics. It has a variety of applications due to its generality. It can be applied to graphs, hypergraphs, oriented graphs, and other combinatorial structures~\cite{Balogh2017RT,Hladk2016,Choi2020,Balogh2022}.
Here we give a very brief introduction adapted for drawings of complete bipartite graphs by stating the basic definitions and claims. Proofs of the claims made in this section can be found in \cite{Razborov2007}.

We call a drawing of a graph \emph{good} if \\(a) adjacent edges intersect only at their common endvertex;\\ (b) any two edges have at most one intersection;\\ (c) intersections of two nonadjacent edges are crossings, not tangential; and \\ (d) no three edges cross at a single point.

Crossing-minimal drawings can be chosen to be good. Since we investigate crossing-minimal drawings of graphs and their induced subdrawings, we assume that all drawings are good.

\newcommand{\K}{K^{D}}

To streamline the exposition, we assume that complete bipartite graphs are vertex 2-colored, where one part is red and the other part is blue. Notation $\K_{r,b}$ means a drawing of $K_{r,b}$ where the red part has $r$ vertices and the blue part has $b$ vertices. We use $V_r(\K_{r,b})$ and $V_b(\K_{r,b})$ to denote the set of red and blue vertices, respectively.
The following text uses the notion of isomorphism of two drawings of graphs that we specify later. The choice of isomorphism is not important for the theory but it is crucial for applications. For now we require that there are only finitely many non-isomorphic drawings of $K_{r,b}$ for each $r,b \geq 0$.

Let $\K_{r_1,b_1}$ and $\K_{r_2,b_2}$ be two drawings, where $r_1 \leq r_2$ and $b_1 \leq b_2$. 
Let $X_r \subseteq V_r(\K_{r_2,b_2})$ and $X_b \subseteq V_b(\K_{r_2,b_2})$ be subsets of red and blue vertices respectively chosen independently and uniformly at random from the red and blue vertex sets, respectively, where $|X_r|=r_1$ and $|X_b|=b_1$. 
Denote by $P(\K_{r_1,b_1}, \K_{r_2,b_2})$ the probability that $\K_{r_2,b_2}[X_r \cup X_b]$ is isomorphic to $\K_{r_1,b_1}$. 
This is also a density of $\K_{r_1,b_1}$ in $\K_{r_2,b_2}$.

Let $\mathcal{F}$ denote the collection of all drawings of all complete bipartite graphs up to isomorphism. 
Let $\mathcal{F}_{r,b}$ denote those drawings in $\mathcal{F}$ with red and blue parts of size $r$ and $b$, respectively. Let $\mathbb{R}\mathcal{F}$ denote the set of all formal finite linear combinations  of drawings in  $\mathcal{F}$ with real coefficients.
Let $\mathcal{K}$ be a linear subspace generated by all linear combinations of
\[
\K_{r,b} -  \sum_{\K \in \mathcal{F}_{r',b'}} P(\K_{r,b}, \K)\K,
\]
where $r' \geq r$ and $b' \geq b$.

The algebra $\mathcal{A}$ is the space $\mathbb{R}\mathcal{F}$ factorized by $\mathcal{K}$.
Observe that addition and multiplication by a real number are naturally defined in $\mathcal{A}$.
Next we introduce a multiplication of two drawings in $\mathcal{F}$.
Let $\K_{r_1,b_1}, \K_{r_2,b_2} \in \mathcal{F}$. Fix one $\K \in \mathcal{F}_{r_1+r_2, b_1+b_2}$.
Let $X_{r_1} \subseteq V_r(\K)$ and $X_{b_1}  \subseteq V_b(\K)$ 
be subsets picked uniformly at random, where $|X_{r_1}|=r_1$ and $|X_{b_1}|=b_1$. Let $X_{r_2} = V_r(\K) \setminus X_{r_1}$ and $X_{b_2} = V_b(\K) \setminus X_{b_1}$.
Denote by $P(\K_{r_1,b_1}, \K_{r_2,b_2}; \K)$ the probability that $\K[X_{r_1}\cup X_{b_1}]$ is isomorphic to  $\K_{r_1,b_1}$ and $\K[X_{r_2}\cup X_{b_2}]$ is isomorphic to  $\K_{r_2,b_2}$. 
Define
\begin{align}
\label{eq:multiplicaton}
\K_{r_1,b_1} \cdot \K_{r_2,b_2} = \sum_{\K \in \mathcal{F}_{r_1+r_2, b_1+b_2}} P(\K_{r_1,b_1}, \K_{r_2,b_2}; \K) \K.
\end{align}
This multiplication has a unique linear extension to $\mathbb{R}\mathcal{F}$ which yields a well-defined multiplication in $\mathcal{A}$.
This is proved formally in~\cite[Lemma 2.4]{Razborov2007}.
Notice that a multiplication of drawings is a linear combination of some other drawings on more vertices.

For each $\sigma$ in $\mathcal{F}$, define all drawings with a fixed labeled copy of $\sigma$ by $\mathcal{F}^\sigma$.
It is possible to define an algebra $\mathcal{A}^\sigma$ built from $\mathcal{F}^\sigma$ in a similar way as $\mathcal{A}$ is built from $\mathcal{F}$. 
See \cite{BaloghHill2019} or \cite{Razborov2007} for more details.

A sequence of drawings of graphs $(\K_{n,n})_{n \geq 1}$ is said to be \emph{convergent} if for every drawing $\K \in \mathcal{F}$ the following limit exists
\[
\lim_{n\to\infty} P(\K, \K_{n,n}).
\]
A compactness argument using Tychonoff’s theorem yields that every infinite sequence has a convergent subsequence. 
Let $(\K_{n,n})_{n \geq 1}$ be a convergent sequence.
Let $\phi(\K) := \lim_{n\to\infty} P(K^D, \K_{n,n})$ for every $\K \in \mathcal{F}$
and extend $\phi$ to $\mathcal{A}$ by linearity. 
The mapping $\phi$ can be seen as the \emph{limit of the sequence} $(\K_{n,n})_{n \geq 1}$.
The resulting mapping $\phi: \mathcal{A} \to \mathbb{R}$ is actually a homomorphism
satisfying $\phi(\K) \in [0,1]$ for every $\K \in \mathcal{F}$.

Let $\Hom$ be the set of all unital algebra homomorphisms $\psi$ from the algebra $\mathcal{A}$ to $\mathbb{R}$ such that $\psi(\K) \geq 0$ for every $\K \in \mathcal{F}$. This set of homomorphisms is exactly the set of all limits of convergent sequences in $\mathcal{F}$, see \cite[Theorem 3.3]{Razborov2007}.

Similarly to $\Hom$, it is possible to define $\text{Hom}^+(\mathcal{A}^\sigma,\mathbb{R})$ as limits of convergent sequences where each drawing has a fixed labeled subdrawing of $\sigma$.
Finally, there is a linear \emph{unlabeling operator}
$\left\llbracket . \right\rrbracket_\sigma : \mathcal{A}^\sigma \to \mathcal{A}$ such that if $A \in \mathcal{A}^\sigma$ satisfies $\psi(A) \geq 0$ for all $\psi \in \text{Hom}^+(\mathcal{A}^\sigma, \mathbb{R})$, then $\phi(\left\llbracket A \right\rrbracket_\sigma)\geq 0$ for all $\phi \in \text{Hom}^+(\mathcal{A}, \mathbb{R})$.

The proofs in the paper are constructed as follows.
For fixed $r$ and $b$ we want to find a constant $C$ such that for all $\phi \in \Hom$
\begin{align}\label{eq:C}
\phi\left(  \sum_{F \in \mathcal{F}_{r,b}} \ucr{F} \, F \right) \geq C.
\end{align}
Letting $C :=  \min_{F \in \mathcal{F}_{r,b}} \ucr{F}$ would work but it may be improved.
We do it using sums of squares obtained in $\mathcal{A}^\sigma$ for various  $\sigma$.
Denote by $T$ a set of types, some labeled drawings from  $\mathcal{F}$.
We do the following for each $\sigma \in T$.
Pick $r',b' \geq 1$ such that 
$2r'-|V_r(\sigma)| = r$ and $2b'-|V_b(\sigma)| = b$. 
This is necessary since we will perform multiplications and we need the result as a linear combination in $\mathcal{F}_{r,b}$; recall \eqref{eq:multiplicaton}.
Let $n_\sigma = |\mathcal{F}^\sigma_{r',b'}|$ and let $F^\sigma$ be a vector with an arbitrary ordering of elements from $\mathcal{F}^\sigma_{r',b'}$.
Then for every $\phi \in \Hom$ and a positive semidefinite matrix $M_\sigma \in \mathbb{R}^{n_\sigma \times n_\sigma}$ the following holds
\begin{align}\label{eq:M}
0 \leq \phi\left(  \left\llbracket  (F^\sigma)^T M_\sigma  F^\sigma \right\rrbracket_\sigma  \right) 
=
\phi\left( 
\sum_{F \in \mathcal{F}_{r,b}} c_F(M_{\sigma},\sigma) \, F
 \right),
\end{align}
where $c_F(M_\sigma,\sigma)$  are real numbers depending on $F$, $M_\sigma$ and $\sigma$.
Since some entries in $M_\sigma$ may be negative, also some $c_F(M_\sigma,\sigma)$ may be negative. This is what makes the approach work.
Equations \eqref{eq:C} and \eqref{eq:M} can be combined
since they are linear combinations in $\mathcal{F}_{r,b}$.
Using linearity of $\phi$ we get
\begin{align*}
\phi\left(  \sum_{F \in \mathcal{F}_{r,b}} \ucr{F} \,  F \right) 
&\geq 
\phi\left(  \sum_{F \in \mathcal{F}_{r,b}} \ucr{F} \, F \right) 
- 
\underbrace{
\sum_{\sigma \in T} \phi\left(  \left\llbracket  (F^\sigma)^T M_\sigma  F^\sigma \right\rrbracket_\sigma  \right)
}_{\geq 0}
=
\phi\left(  \sum_{F \in \mathcal{F}_{r,b}} c_F F \right)
\\
&\geq
\min_{F \in \mathcal{F}_{r,b}}(c_F) \cdot 
\underbrace{\phi\left(  \sum_{F \in \mathcal{F}_{r,b}}  F \right)}_{=1} 
=
\min_{F \in \mathcal{F}_{r,b}}(c_F) 
=
C,
\end{align*}
where the coefficients $c_F$ are functions of the $M_\sigma$s since they are obtained as
$
c_F := \ucr{F} - \sum_{\sigma} c_F(M_\sigma,\sigma)
$.
The goal is to obtain the largest possible $C$, which means solving
\begin{align}\label{eq:sdp}
\max_{M_\sigma \succeq 0, \sigma\in T}\min_{F \in \mathcal{F}_{r,b}}(c_F). 
\end{align}
This is a semidefinite program that can be numerically solved. We used solvers CSDP~\cite{Borchers1999} and Mosek~\cite{Mosek}.
Notice that if one proves for some $A \in \mathcal{A}$ that
$\phi(A) \geq 0$
for every $\phi \in \Hom$, then
$-y \cdot \phi(A)$ for a parameter $y \geq 0$ can be added to the calculation. 

Finally, in order to obtain a rigorous bound, the numerical solution needs to be rounded to a rational solution. It can be done using the software we developed in SageMath~\cite{sagemath}. For more details about ideas behind the rounding see~\cite{Balogh2014,Baber2011}.
The proofs are certified by the positive semidefinite the matrices $M_\sigma$s or decompositions as sums-of-squares that can be obtained at \url{http://lidicky.name/pub/knn/}. There is no need to solve a semidefinite programs in order to verify the solution.

The machinery of flag algebras works with limits of convergent sequences. 
The same machinery can be applied also in $K^D_{n,n}$ but multiplication \eqref{eq:multiplicaton} will come with an additive error $O(1/n)$. Since the number of multiplications in our calculations is finite, the limit result can be transferred to a result about a (sufficiently large) graph at the cost of an $O(1/n)$ error.

\subsection{Isomorphism}

In order to apply the method of flag algebras to crossing numbers, 
we need to define when two drawings are isomorphic.
Let $G^D$ be a drawing of a graph.
The following types of information can be extracted from $G^D$.
\begin{itemize}
    \item[(C1)] For each pair of edges $e_1$ and $e_2$, whether or not they cross each other.
    \item[(C2)] For each quadruple $Q \subset V(G^D)$, whether  
     $G^D[Q]$ induces a crossing or not.
\end{itemize}
In case of a rectilinear drawing, the following information can be also extracted from $G^D$.
\begin{itemize}
    \item[(C3)] For vertices $u_1,u_2,u_3,u_4$ whether the line  $u_1u_2$ crosses the segment $u_3u_4$.
    \item[(C4)] For each ordered triple $u_1,u_2,u_3$ if $u_3$ is on the left or right of the line $u_1u_2$.
    \item[(C5)] For each vertex the cyclic clockwise order of its neighbors (that is, the {\em rotation} at the vertex). 
\end{itemize}
The lists above are not exhaustive and there may be other possibilities of what information to extract from $G^D$.
Notice that distinct drawings can result in identical configurations in $\mathcal{F}$. 
Typically preserving more information about the drawings gives stronger bounds. On the other hand, flag algebra method needs to work with a list of all drawings in $\mathcal{F}_{r,b}$ for some fixed $r,b$, and larger $r,b$ usually give better bounds. For computational purposes, $|\mathcal{F}_{r,b}|$ needs to be under 200,000 for solving the resulting semidefinite program with CSDP. It is not clear if it is better to have smaller $r,b$ with configurations keeping more information from the drawing or larger $r,b$ and configurations with less information.  
With a new solver~\cite{lpsdp2026}, we were able to obtain bounds on instances where $|\mathcal{F}_{r,b}|$ is a few million.

\section{Generating General Drawings}\label{sec:generate}

For the flag algebra calculations we need all small configurations that can
occur in good drawings of complete bipartite graphs.  We first generate the
drawings themselves, identifying two drawings when a color-preserving
homeomorphism of the sphere maps one to the other.  In particular, no face is
distinguished as an outer face.  Once the drawings have been generated, we
retain only the combinatorial information required for a particular
calculation, such as (C1) or (C2), and reduce the resulting configurations
further up to isomorphism.  Thus distinct drawings at the first stage may give
the same configuration at the second stage.

A program for generating such drawings was developed by Norin and Zwols~\cite{Zwols2014}. 
The program works well for generating drawings of $K_{3,4}$ but fails after two days of running when generating drawings of $K_{4,4}$. Their program and ours match on the number of distinct drawings up to $K_{3,4}$ after accounting for isomorphism via vertex relabeling.
Since we failed to fix the program, we implemented our own version and describe the principles below, since other researchers may find it interesting. ChatGPT-6 was able to identify the mistake which was an integer overflow. 
However, the program has not finished running within a week.

The construction begins with $K_{2,2}$.  The algorithm describes a general process for extending a drawing of a graph to drawings of a supergraph.  We start with $K_{2,2}$ because it is a subgraph of our desired graphs and its complete set of good drawings is easy to describe.  Up to homeomorphism of the sphere
there are two good drawings of $K_{2,2}$, according to whether or not the
drawing contains a crossing; see Figure~\ref{fig:k22-initial}.  Larger
drawings are obtained recursively by adding one vertex at a time, followed by
all edges incident with that vertex.

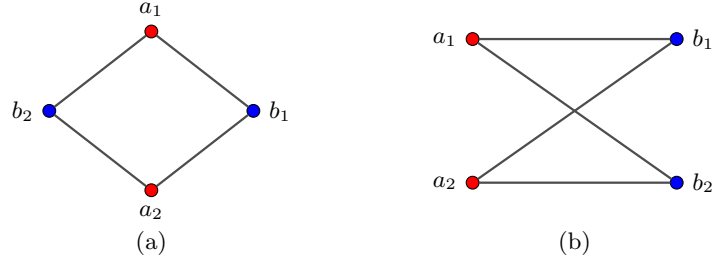
\begin{figure}[htbp]
  \centering
  \begin{tikzpicture}[
      every node/.style={font=\small},
      oldedge/.style={draw=black!70,line width=.85pt}
    ]
    \begin{scope}[xshift=-2.8cm]
      \node[vertexA,label=above:$a_1$] (pa1) at (0,1.05) {};
      \node[vertexB,label=right:$b_1$] (pb1) at (1.35,0) {};
      \node[vertexA,label=below:$a_2$] (pa2) at (0,-1.05) {};
      \node[vertexB,label=left:$b_2$] (pb2) at (-1.35,0) {};
      \draw[oldedge] (pa1)--(pb1)--(pa2)--(pb2)--(pa1);
      \node at (0,-1.75) {(a)};
    \end{scope}

    \begin{scope}[xshift=2.8cm]
      \node[vertexA,label=left:$a_1$] (ca1) at (-1.35,.95) {};
      \node[vertexA,label=left:$a_2$] (ca2) at (-1.35,-.95) {};
      \node[vertexB,label=right:$b_1$] (cb1) at (1.35,.95) {};
      \node[vertexB,label=right:$b_2$] (cb2) at (1.35,-.95) {};
      \draw[oldedge] (ca1)--(cb1);
      \draw[oldedge] (ca2)--(cb2);
      \draw[oldedge] (ca1)--(cb2);
      \draw[oldedge] (ca2)--(cb1);
      \node at (0,-1.75) {(b)};
    \end{scope}
  \end{tikzpicture}
  \caption{The two good drawings of $K_{2,2}$ up to homeomorphism of the
  sphere: (a) the planar drawing and (b) the drawing with one crossing.}
  \label{fig:k22-initial}
\end{figure}

Suppose that $G^D$ is a drawing that has already been constructed.  We work
with its \emph{planarization} $P(G^D)$, obtained by replacing every crossing of
$G^D$ by a vertex of degree four.  Officially, $P(G^D)$ is a graph, and a particular
drawing should be written as $P(G^D)^D$. Since we never consider any drawing of $P(G^D)$
other than the one that agrees with the underlying $G^D$, we drop the second superscript
and use $P(G^D)$ to refer to the drawing. Each original edge is then  represented 
by a path in $P(G^D)$.  We call the portions of these paths between consecutive 
vertices of $P(G^D)$ \emph{elementary segments}.  The faces of $P(G^D)$ are
represented by their boundary walks rather than by simple cycles, since the boundary of 
a face may encounter the same vertex or elementary segment more than once.

To insert a new edge $xy$, we consider the face-dual of $P(G^D)$.  Its vertices
are the faces of $P(G^D)$, and crossing an elementary segment gives an adjacency
between the two face-sides incident with that segment.  A route for $xy$ is
therefore described by an alternating sequence
\[
   F_0,s_1,F_1,s_2,\ldots,s_k,F_k,
\]
where $s_i$ is an elementary segment incident with both $F_{i-1}$ and $F_i$.
Crossing $s_i$ means that the new edge crosses the original edge containing
$s_i$.

Figure~\ref{fig:edge-routing} illustrates this operation.  A new blue vertex
$x$ is placed in the upper face of the planarization of the crossed drawing of
$K_{2,2}$.  The route in panel~(a) crosses the elementary segments $s_1$ and
$s_2$, passing successively through the faces $F_0,F_1,F_2$.  Panel~(b) shows
the corresponding insertion of the edge $xa_2$.  The two crossings become new
vertices $q_1$ and $q_2$ of the planarization, and the crossed elementary
segments are subdivided accordingly.

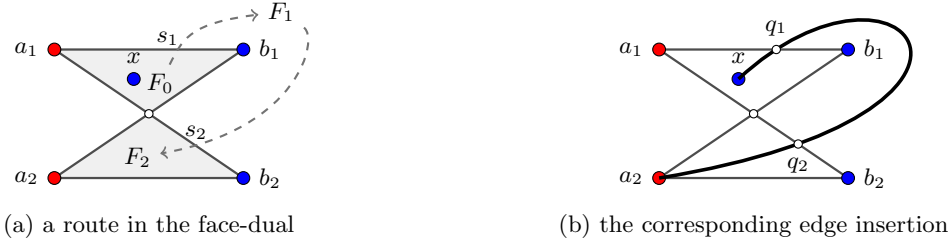
\begin{figure}[htbp]
  \centering
  \begin{tikzpicture}[
      every node/.style={font=\small},
      oldedge/.style={draw=black!70,line width=.85pt},
      newedge/.style={draw=black,line width=1.35pt},
      dual/.style={draw=black!55,dashed,line width=.8pt,->},
      crossingvertex/.style={
        circle,draw=black,fill=white,inner sep=1.1pt
      },
      facelabel/.style={font=\small\itshape},
      faceshade/.style={fill=black!6}
    ]

    \begin{scope}[xshift=-4.0cm]
      \fill[faceshade]
        (-1.25,.85)--(1.25,.85)--(0,0)--cycle;
      \fill[faceshade]
        (-1.25,-.85)--(1.25,-.85)--(0,0)--cycle;

      \draw[oldedge] (-1.25,.85)--(1.25,.85);
      \draw[oldedge] (-1.25,-.85)--(1.25,-.85);
      \draw[oldedge] (-1.25,.85)--(0,0)--(1.25,-.85);
      \draw[oldedge] (-1.25,-.85)--(0,0)--(1.25,.85);

      \node[vertexA,label=left:$a_1$] at (-1.25,.85) {};
      \node[vertexA,label=left:$a_2$] at (-1.25,-.85) {};
      \node[vertexB,label=right:$b_1$] at (1.25,.85) {};
      \node[vertexB,label=right:$b_2$] at (1.25,-.85) {};
      \node[crossingvertex] at (0,0) {};

      \node[vertexB,label=above:$x$] at (-0.2,.46) {};

      \coordinate (s1) at (.25,.75);
      \coordinate (s2) at (.35,-.04);

      \node[above=1pt] at (s1) {$s_1$};
      \node[below right=0pt] at (s2) {$s_2$};

      \node[facelabel] (F0) at (.15,.4) {$F_0$};
      \node[facelabel] (F1) at (1.75,1.35) {$F_1$};
      \node[facelabel] (F2) at (-.15,-.58) {$F_2$};

      \draw[dual] (F0) .. controls (0.45,1) and (0.55, 1.2) .. (F1);
      \draw[dual] (F1) .. controls (2.25,1) and (2,-0.2)  .. (F2);

      \node at (0,-1.48) {(a) a route in the face-dual};
    \end{scope}

    \begin{scope}[xshift=4.0cm]
      \coordinate (q1) at (.3,.85);
      \coordinate (q2) at (.59,-.4);

      \draw[oldedge] (-1.25,.85)--(q1)--(1.25,.85);
      \draw[oldedge] (-1.25,-.85)--(1.25,-.85);
      \draw[oldedge] (-1.25,.85)--(0,0)--(q2)--(1.25,-.85);
      \draw[oldedge] (-1.25,-.85)--(0,0)--(1.25,.85);

      \node[vertexA,label=left:$a_1$] at (-1.25,.85) {};
      \node[vertexA,label=left:$a_2$] at (-1.25,-.85) {};
      \node[vertexB,label=right:$b_1$] at (1.25,.85) {};
      \node[vertexB,label=right:$b_2$] at (1.25,-.85) {};
      \node[crossingvertex] at (0,0) {};

      \node[vertexB,label=above:$x$] at (-0.2,.46) {};

      \draw[newedge]
        (-0.2,.46) .. controls (2,2.5) and (4,0) .. (-1.25,-.85);

      \node[crossingvertex,label=above:$q_1$] at (q1) {};
      \node[crossingvertex,label=below:$q_2$] at (q2) {};

      \node at (0,-1.48) {(b) the corresponding edge insertion};
    \end{scope}
  \end{tikzpicture}
  \caption{Routing a new edge through the planarization.  In (a), the dashed
  arrows show the face-dual route $F_0,s_1,F_1,s_2,F_2$.  In (b), the edge
  $xa_2$ is drawn along this route.  Crossing $s_1$ and $s_2$ creates the new
  crossing vertices $q_1$ and $q_2$ and subdivides the corresponding
  elementary segments.}
  \label{fig:edge-routing}
\end{figure}

We enumerate all such routes subject to the conditions imposed by goodness.
The new edge may not cross an old edge incident with either of its
endvertices, and it may cross each old edge at most once.  Thus the route may
not use an elementary segment of an adjacent edge, and it may use at most one
elementary segment belonging to any fixed old edge.  The route crosses only
the interiors of elementary segments and never passes through a vertex of the
planarization.  The sequence of faces itself need not be simple: a route may
leave a face and later return to it.

A route in the face-dual records which elementary segments are crossed and in
what order, but this information alone does not always determine an embedding.
For each candidate route we therefore insert the new crossing vertices into
the corresponding boundary walks and split the traversed faces successively
along the new edge.  When a face is visited more than once, a later split must
be compatible with the splits already made in that face.  We discard the
route if no compatible collection of boundary walks remains, and retain every
compatible outcome otherwise.  In particular, the construction never relies
on a purely formal list of crossings: every retained extension is represented
by an actual planarization.

Figure~\ref{fig:worked-k24} shows one branch of the recursion from $K_{2,2}$
to $K_{2,4}$.  Starting from the crossed drawing of $K_{2,2}$, we place a new
blue vertex $b_3$ in a face and insert its two incident edges.  The second of
these crosses $a_1b_1$, creating the new crossing vertex $p$ and subdividing
that elementary segment.  We then place $b_4$ in a face of the resulting
drawing of $K_{2,3}$ and insert its two incident edges.  In the final
insertion, $a_1b_4$ crosses both the previously added edge $a_2b_3$ and the
old edge $a_2b_2$.  This example is only one possible branch of the
construction; all admissible choices are considered by the program.

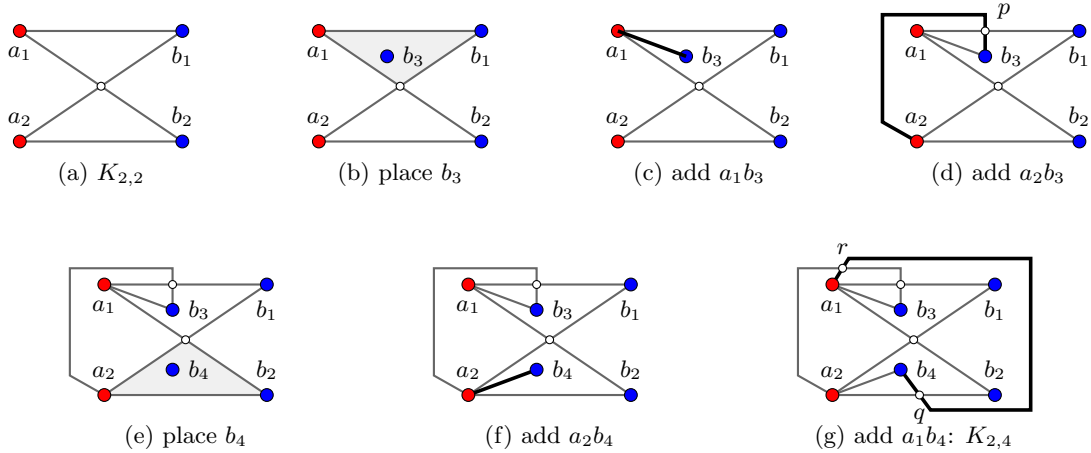
\begin{figure}[htbp]
\centering
\begin{tikzpicture}[
    scale=.86,
    every node/.style={font=\small},
    oldedge/.style={draw=black!60,line width=.8pt},
    newedge/.style={draw=black,line width=1.35pt},
    crossingvertex/.style={
      circle,draw=black,fill=white,inner sep=1.05pt
    },
    faceshade/.style={fill=black!6},
    panellabel/.style={font=\small,align=center}
  ]

\begin{scope}[xshift=-6.9cm,yshift=2.45cm]
  \draw[oldedge] (-1.25,.85)--(1.25,.85);
  \draw[oldedge] (-1.25,-.85)--(1.25,-.85);
  \draw[oldedge] (-1.25,.85)--(0,0)--(1.25,-.85);
  \draw[oldedge] (-1.25,-.85)--(0,0)--(1.25,.85);

  \node[vertexA,label=below:$a_1$] at (-1.25,.85) {};
  \node[vertexA,label=above:$a_2$] at (-1.25,-.85) {};
  \node[vertexB,label=below:$b_1$] at (1.25,.85) {};
  \node[vertexB,label=above:$b_2$] at (1.25,-.85) {};
  \node[crossingvertex] at (0,0) {};

  \node[panellabel] at (0,-1.35) {(a) $K_{2,2}$};
\end{scope}

\begin{scope}[xshift=-2.3cm,yshift=2.45cm]
  \fill[faceshade]
    (-1.25,.85)--(1.25,.85)--(0,0)--cycle;

  \draw[oldedge] (-1.25,.85)--(1.25,.85);
  \draw[oldedge] (-1.25,-.85)--(1.25,-.85);
  \draw[oldedge] (-1.25,.85)--(0,0)--(1.25,-.85);
  \draw[oldedge] (-1.25,-.85)--(0,0)--(1.25,.85);

  \node[vertexA,label=below:$a_1$] at (-1.25,.85) {};
  \node[vertexA,label=above:$a_2$] at (-1.25,-.85) {};
  \node[vertexB,label=below:$b_1$] at (1.25,.85) {};
  \node[vertexB,label=above:$b_2$] at (1.25,-.85) {};
  \node[crossingvertex] at (0,0) {};

  \node[vertexB,label=right:$b_3$] at (-0.2,.46) {};

  \node[panellabel] at (0,-1.35) {(b) place $b_3$};
\end{scope}

\begin{scope}[xshift=2.3cm,yshift=2.45cm]
  \draw[oldedge] (-1.25,.85)--(1.25,.85);
  \draw[oldedge] (-1.25,-.85)--(1.25,-.85);
  \draw[oldedge] (-1.25,.85)--(0,0)--(1.25,-.85);
  \draw[oldedge] (-1.25,-.85)--(0,0)--(1.25,.85);

  \node[vertexA,label=below:$a_1$] at (-1.25,.85) {};
  \node[vertexA,label=above:$a_2$] at (-1.25,-.85) {};
  \node[vertexB,label=below:$b_1$] at (1.25,.85) {};
  \node[vertexB,label=above:$b_2$] at (1.25,-.85) {};
  \node[crossingvertex] at (0,0) {};
  \node[vertexB,label=right:$b_3$] at (-0.2,.46) {};

  \draw[newedge] (-1.25,.85)--(-0.2,.46);

  \node[panellabel] at (0,-1.35) {(c) add $a_1b_3$};
\end{scope}

\begin{scope}[xshift=6.9cm,yshift=2.45cm]
  \coordinate (p) at (-0.2,.85);

  \draw[oldedge] (-1.25,.85)--(p)--(1.25,.85);
  \draw[oldedge] (-1.25,-.85)--(1.25,-.85);
  \draw[oldedge] (-1.25,.85)--(0,0)--(1.25,-.85);
  \draw[oldedge] (-1.25,-.85)--(0,0)--(1.25,.85);
  \draw[oldedge] (-1.25,.85)--(-0.2,.46);

  \draw[newedge]
    (-1.25,-.85)--(-1.78,-.55)--(-1.78,1.1)
    --(-0.2,1.1)--(p)--(-0.2,.46);

  \node[vertexA,label=below:$a_1$] at (-1.25,.85) {};
  \node[vertexA,label=above:$a_2$] at (-1.25,-.85) {};
  \node[vertexB,label=below:$b_1$] at (1.25,.85) {};
  \node[vertexB,label=above:$b_2$] at (1.25,-.85) {};
  \node[vertexB,label=right:$b_3$] at (-0.2,.46) {};

  \node[crossingvertex] at (0,0) {};
  \node[crossingvertex,label=above right:$p$] at (p) {};

  \node[panellabel] at (0,-1.35) {(d) add $a_2b_3$};
\end{scope}

\begin{scope}[xshift=-5.6cm,yshift=-1.45cm]
  \coordinate (p) at (-0.2,.85);

  \fill[faceshade]
    (-1.25,-.85)--(1.25,-.85)--(0,0)--cycle;

  \draw[oldedge] (-1.25,.85)--(p)--(1.25,.85);
  \draw[oldedge] (-1.25,-.85)--(1.25,-.85);
  \draw[oldedge] (-1.25,.85)--(0,0)--(1.25,-.85);
  \draw[oldedge] (-1.25,-.85)--(0,0)--(1.25,.85);

  \draw[oldedge] (-1.25,.85)--(-0.2,.46);
  \draw[oldedge]
    (-1.25,-.85)--(-1.78,-.55)--(-1.78,1.1)
    --(-0.2,1.1)--(p)--(-0.2,.46);

  \node[vertexA,label=below:$a_1$] at (-1.25,.85) {};
  \node[vertexA,label=above:$a_2$] at (-1.25,-.85) {};
  \node[vertexB,label=below:$b_1$] at (1.25,.85) {};
  \node[vertexB,label=above:$b_2$] at (1.25,-.85) {};
  \node[vertexB,label=right:$b_3$] at (-0.2,.46) {};
  \node[vertexB,label=right:$b_4$] at (-0.2,-.46) {};

  \node[crossingvertex] at (0,0) {};
  \node[crossingvertex] at (p) {};

  \node[panellabel] at (0,-1.5) {(e) place $b_4$};
\end{scope}

\begin{scope}[xshift=0cm,yshift=-1.45cm]
  \coordinate (p) at (-0.2,.85);

  \draw[oldedge] (-1.25,.85)--(p)--(1.25,.85);
  \draw[oldedge] (-1.25,-.85)--(1.25,-.85);
  \draw[oldedge] (-1.25,.85)--(0,0)--(1.25,-.85);
  \draw[oldedge] (-1.25,-.85)--(0,0)--(1.25,.85);

  \draw[oldedge] (-1.25,.85)--(-0.2,.46);
  \draw[oldedge]
    (-1.25,-.85)--(-1.78,-.55)--(-1.78,1.1)
    --(-0.2,1.1)--(p)--(-0.2,.46);

  \draw[newedge] (-1.25,-.85)--(-0.2,-.46);

  \node[vertexA,label=below:$a_1$] at (-1.25,.85) {};
  \node[vertexA,label=above:$a_2$] at (-1.25,-.85) {};
  \node[vertexB,label=below:$b_1$] at (1.25,.85) {};
  \node[vertexB,label=above:$b_2$] at (1.25,-.85) {};
  \node[vertexB,label=right:$b_3$] at (-0.2,.46) {};
  \node[vertexB,label=right:$b_4$] at (-0.2,-.46) {};

  \node[crossingvertex] at (0,0) {};
  \node[crossingvertex] at (p) {};

  \node[panellabel] at (0,-1.5) {(f) add $a_2b_4$};
\end{scope}

\begin{scope}[xshift=5.6cm,yshift=-1.45cm]
  \coordinate (p) at (-0.2,.85);
  \coordinate (r) at (-1.09375,1.1);
  \coordinate (q) at (.09,-.85);

  \draw[oldedge] (-1.25,.85)--(p)--(1.25,.85);
  \draw[oldedge] (-1.25,-.85)--(q)--(1.25,-.85);
  \draw[oldedge] (-1.25,.85)--(0,0)--(1.25,-.85);
  \draw[oldedge] (-1.25,-.85)--(0,0)--(1.25,.85);

  \draw[oldedge] (-1.25,.85)--(-0.2,.46);

  \draw[oldedge]
    (-1.25,-.85)--(-1.78,-.55)--(-1.78,1.1)
    --(r)--(-0.2,1.1)--(p)--(-0.2,.46);

  \draw[oldedge] (-1.25,-.85)--(-0.2,-.46);

  \draw[newedge]
    (-1.25,.85)--(r)--(-1.0,1.25)--(1.8,1.25)
    --(1.8,-1.08)--(.26,-1.08)--(q)--(-0.2,-.46);

  \node[vertexA,label=below:$a_1$] at (-1.25,.85) {};
  \node[vertexA,label=above:$a_2$] at (-1.25,-.85) {};
  \node[vertexB,label=below:$b_1$] at (1.25,.85) {};
  \node[vertexB,label=above:$b_2$] at (1.25,-.85) {};
  \node[vertexB,label=right:$b_3$] at (-0.2,.46) {};
  \node[vertexB,label=right:$b_4$] at (-0.2,-.46) {};

  \node[crossingvertex] at (0,0) {};
  \node[crossingvertex] at (p) {};
  \node[crossingvertex,label=above:$r$] at (r) {};
  \node[crossingvertex,label=below:$q$] at (q) {};

  \node[panellabel] at (0,-1.5)
    {(g) add $a_1b_4$: $K_{2,4}$};
\end{scope}

\end{tikzpicture}

\caption{One branch of the recursive construction.  In (a)--(d), a new blue
vertex $b_3$ is placed in a face of a drawing of $K_{2,2}$ and its two
incident edges are inserted successively.  The edge $a_2b_3$ crosses
$a_1b_1$ at $p$, so the crossed elementary segment is subdivided there.  In
(e)--(g), the same procedure is repeated with a new blue vertex $b_4$.  The
last edge $a_1b_4$ crosses $a_2b_3$ at $r$ and $a_2b_2$ at $q$.  Thus the
last panel is one of the drawings of $K_{2,4}$ obtained by the recursion.}
\label{fig:worked-k24}
\end{figure}

To pass from $\K_{r,b}$ to $\K_{r+1,b}$, we introduce a new red vertex and
then add its $b$ incident edges one at a time by the procedure above.  The
construction of $\K_{r,b+1}$ is symmetric.  After all edges incident with the
new vertex have been inserted, we retain one representative from each
homeomorphism class on the sphere before proceeding to the next stage.

The construction is exhaustive.  Let $D'$ be any good drawing at the next
stage, and delete the edges incident with the last-added vertex in reverse
order.  Immediately before an edge $xy$ is deleted, its interior passes
through a sequence of faces of the planarization of the remaining drawing.
Since $D'$ is good, this route crosses no edge incident with $x$ or $y$ and
crosses every other old edge at most once.  It is therefore among the routes
considered by the algorithm.  Moreover, deleting $xy$ reverses exactly the
successive subdivisions of elementary segments and facial boundary walks
performed during insertion, so the corresponding compatibility test
succeeds.  Repeating this for all edges incident with the deleted vertex shows
that every drawing at the next stage arises from one at the preceding stage.
Induction from the two drawings of $K_{2,2}$ therefore shows that the
procedure produces every good drawing required for our calculations.

Finally, each generated drawing is converted to the combinatorial structure
used in the relevant flag algebra calculation.  Under (C1), for example, we
retain for every pair of edges whether or not they cross; under (C2), we
retain only which sets of two red and two blue vertices induce a crossing.
The resulting colored structures are then reduced up to isomorphism.  Since
this step forgets some of the topological information in the drawing, several
distinct homeomorphism classes may give the same configuration.  The
corresponding partially labeled structures are the flags used in the flag
algebra calculation.

The number of isomorphism classes for some small graphs for some kinds of isomorphism 
are summarized in Table~\ref{tab:class-counts}. All calculations in Table~\ref{tab:class-counts} were made on Iowa State University's Nova cluster; excluding $K_{3,6}$, these calculations took a cumulative 8.6 CPU-days, about 16 GB of memory, and about 20 GB of storage space. The memory is primarily due to parallel processing; the calculations are embarrassingly parallel, so the whole table excluding $K_{3,6}$ was created in about 8 hours of wall-time using 32 CPU cores. The $K_{3,6}$ calculation alone took 78.5 CPU-days, around 41 GB of memory due to parallel processing, and about 50 GB of storage; this was parallelized down to 30 wall-time hours using 64 CPU cores. This program for enumerating drawings of general graphs can be found at \url{https://github.com/nickLayman/Knn-drawing-generation}.

\begin{table}[htbp]
    \centering
    \resizebox{\textwidth}{!}{
    \begin{tabular}{|c|c|c|c|c|c|c|} \hline
        Graph & Drawings           & Drawings     & (C1) flags         & (C1) flags   & (C2) flags         & (C2) flags \\
              & (color preserving) & (colorblind) & (color preserving) & (colorblind) & (color preserving) & (colorblind) \\ \hline
        $K_{2,2}$ & 2 & 2 & 2 & 2 & 2 & 2 \\ \hline
        $K_{2,3}$ & 6 & - & 5 & - & 4 & - \\ \hline
        $K_{2,4}$ & 40 & - & 21 & - & 11 & - \\ \hline
        $K_{2,5}$ & 394 & - & 144 & - & 33 & - \\ \hline
        $K_{2,6}$ & 5,457 & - & 1,697 & - & 144 & - \\ \hline
        $K_{3,3}$ & 177 & 102 & 112 & 70 & 18 & 13 \\ \hline
        $K_{3,4}$ & 14,823 & - & 6,393 & - & 355 & - \\ \hline
        $K_{3,5}$ & 2,163,813 & - & 468,742 & - & 17,643 & - \\ \hline
        $K_{3,6}$ & 522,392,504 & - & 40,572,486 & - & 1,448,110 & - \\ \hline
        $K_{4,4}$ & 18,031,850 & 9,019,729 & 1,931,529 & 967,452 & 123,926 & 62,415 \\ \hline
    \end{tabular}
    }
    \caption{The cardinality of equivalence classes for various equivalence relations on complete bipartite graphs with at most 8 vertices, and $K_{3,6}$. We exclude any case with a part of size 1; all good drawings of $K_{1,m}$ must be planar and thus there is only one up to isomorphism in any of the above senses. For balanced $K_{m,m}$, one may consider whether swapping the color parts ``should" count as the same drawing or the same flag. For the curious reader, we have included those numbers as well and refer to this case as colorblind. This only changes the cases where the parts are equal size so we exclude the others for clarity.}
    \label{tab:class-counts}
\end{table}

\section{General Drawings}\label{sec:knn}
In this section we describe the proof of Theorem~\ref{thm:Knn}.
The main step is the following lemma which we prove using flag algebras.

\begin{lemma}\label{lem:knnflag44}
Let $n$ be sufficiently large and $K^D$ be any good drawing of $K_{n,n}$.
The proportion of induced $K_{2,2}$ subdrawings of $K^D$ that induce a crossing is at least $\crKnnLemmaCertified$.
\end{lemma}
\begin{proof}
We used isomorphism (C1).  
We obtained all drawings of $K_{4,4}$ up to isomorphisms under (C1), denoted here by $\mathcal{F}_{4,4}$, using a program described in Section~\ref{sec:generate}.
In this setting, $|\mathcal{F}_{4,4}| = 1,931,529$.
We used symmetry reduction developed by Kiem, Pokutta and  Spiegel~\cite{Kiem2026} to reduce the number of constraints in the resulting semidefinite program to 967,452.
Denote by $C$ the crossed drawing of $K_{2,2}$.
Using flag algebras, we obtained after 3,233 hours of computation that for every $\phi \in \Hom$: 
\[
\phi(C) = 
\phi\left(\sum_{\K \in \mathcal{F}_{4,4}} P(C,\K) \cdot\K\right) > 
\crKnnLemmaCertified.
\] 
\end{proof}

\begin{proof}[Proof of Theorem~\ref{thm:Knn}]
Let $K_{n,n}^D$ be a drawing of $K_{n,n}$.

\begin{align*}
 \ucr{K_{n,n}^D} >  \crKnnLemmaCertified \binom{n}{2}\binom{n}{2} 
 \               >  4\cdot \crKnnLemmaCertified  \cdot n^4/16 + o(n^4) > \crKnnCertified \cdot Z(n,n) + o(n^4).
\end{align*}
\end{proof}

\section{Rectilinear Drawings}\label{sec:rectinilear}
In this section we describe the proof of Theorem~\ref{thm:rKnn}.

\begin{lemma}\label{lem:knnflag_rectilinear}
For every sufficiently large $n$, in every rectilinear drawing of $K_{n,n}$ the expected number of crossings in a $K_{4,4}$ chosen uniformly at random is more than 8.8837.
\end{lemma}
\begin{proof}
We used isomorphism (C4). To obtain rectilinear drawings of $K_{4,4}$, we start with all rectilinear drawings of $K_8$ found by Aichholzer~\cite{Aichholzer06} and try all possible splits into $K_{4,4}$. This lead to a program using $112401$ configurations after the same reduction as in Lemma~\ref{lem:knnflag44}.
Using flag algebras, we obtain
\[
\phi\left(\sum_{\K \in \mathcal{F}_{4,4}}\overline{cr}(\K) \K\right) \geq 
\frac{8883777481283804668150619490360970766373630}{1000000000000000000000000000000000000000000}  \approx  8.883777481283804.
\]
\end{proof}

\begin{proof}[Proof of Theorem~\ref{thm:rKnn}]
Let $K_{n,n}^D$ be a rectilinear drawing of $K_{n,n}$.
By double-counting the pairs of crossings and copies of $K_{4,4}$ in $K_{n,n}^D$ containing them  and using Lemma~\ref{lem:knnflag_rectilinear} we get
\begin{align*}
\left(\frac14n^4+o(n^4)\right) \overline{cr}(K_{n,n}^D) &> 8.8837/24^2 n^8 + o(n^8)\\
 \overline{cr}(K_{n,n}^D) &>   8.8837/9  \cdot n^4/16 + o(n^4)
 >  0.987  \cdot Z(n,n) + o(n^4).
\end{align*}
\end{proof}

\section{Tur\'an-Type Restrictions}\label{sec:turan}

In this section we prove Theorem~\ref{thm:tKnn}. It is enough to work with (C2), translating the problem to a 4-uniform hypergraph setting.

\begin{lemma}\label{lem:knnflag_turan}
For sufficiently large $n$, in every drawing of $K_{n,n}$ that does not contain an induced $K_{3,4}$ subdrawing with exactly 2 crossings sharing one vertex, the expected number of crossings in a random $K_{3,4}$ is at least $4.5 - o(1)$.
\end{lemma}
\begin{proof}
We used isomorphism (C2) to work with  $\mathcal{F}_{3,4}$. 
We started with the drawings of $K_{3,4}$ under (C1) and applied (C2) to them. 
This gave 355 drawings up to isomorphism.
Removing the drawings with two crossings sharing exactly one vertex resulted in 
$|\mathcal{F}_{3,4}| = 354$
and a slightly different $\mathcal{A}$.
Using flag algebras, we obtained that for every $\phi \in \Hom$ 
\[
\phi\left(\sum_{\K \in \mathcal{F}_{3,4}}\ucr{\K} \K\right)
\geq \frac{9}{2} .
\]
\end{proof}

\begin{proof}[Proof of Theorem~\ref{thm:tKnn}]
 Let $K_{n,n}^D$ be a drawing of $K_{n,n}$ that does not contain an induced $K_{3,4}$ subdrawing with exactly 2 crossings sharing one vertex.
We proceed by double-counting the pairs of crossings and copies of $K_{3,4}$ in $K_{n,n}^D$ containing them.
Each copy of $K_{2,2}$ in $K_{n,n}^D$ is in $2(n-2)\binom{n-2}{2} = n^3+o(n^3)$ different copies of $K_{3,4}$ in $K_{n,n}^D$.
On the other hand, there are $2\binom{n}3\binom{n}{4}$ copies of $K_{3,4}$ in $K_{n,n}^D$ and by Lemma~\ref{lem:knnflag_turan} the average number of crossings per copy is at least 4.5. This gives
\begin{align*}
(n^3+o(n^3)) \ucr{K_{n,n}^D} &\geq 4.5/72 n^7 + o(n^7)\\
 \ucr{K_{n,n}^D} &\geq n^4/16 + o(n^4).
\end{align*}
\end{proof}

\section{Hill's conjecture}\label{sec:hill}

In this section we describe the proof of Theorem~\ref{thm:hill}.
We use the same set-up as Balogh, Lidick\'y and Salazar~\cite{BaloghHill2019}.
This set-up for flag algebras is significantly different from the other problems mentioned in this paper.
The drawings of $K_n$ are represented by rotation systems. 
That is, for each vertex we remember the clockwise order of the $n-1$ edges to the remaining vertices. 
If there is a drawing corresponding to a rotation system,  then it is called \emph{realizable}. 
The small realizable rotation systems are known~\cite{relizable}.

Denote the set of all realizable rotations systems for $n$ vertices by $\mathcal{E}_n$.
Balogh, Lidick\'y and Salazar~\cite{BaloghHill2019} used $\mathcal{E}_7$, where $|\mathcal{E}_7| = 11,556$.

\begin{lemma}\label{lem:fahill}
    For sufficiently large $n$, let $K_n^D$
    be a drawing of $K_n$. 
    The density of $K_4$ subgraphs that do not induce a crossing is at most $22.007/35$.
\end{lemma}
\begin{proof}
    We used flag algebras with rotation systems on $\mathcal{E}_8$ that has size 5,370,725.
By using lpsdp~\cite{lpsdp2026}, we obtained the desired bound.    
\end{proof}

\begin{proof}[Proof of Theorem~\ref{thm:hill}]
Let $K_n^D$ be a drawing of $K_n$.
By using Lemma~\ref{lem:fahill},
the number of crossings in $K_n^D$ is at least
$
(1-22.007/35) \binom{n}{4}.
$
Hence
\begin{align*}    
\ucr{K_n} \geq  (1-22.007/35) \binom{n}{4} 
          = (1-22.007/35) \cdot \frac{64}{4!} H(n) + o(n^4)
          > 0.9899  H(n) + o(n^4).
\end{align*}
\end{proof}

\section{Conclusion}
Flag algebras are applicable to other variants of the crossing number.
For rectilinear drawings where the red and blue vertices are separated by a straight line we obtained the asymptotic lower bound $0.995 \cdot Z(n,n)$.
For 2-page drawings, in which there is a simple closed curve containing all the vertices but not crossing any edge, we obtained the lower bound $0.976 \cdot Z(n,n)$. Finally, for rectilinear separated 2-page drawings we obtained $0.99998 \cdot Z(n,n)$. This last bound is definitely of interest, since there exist rectilinear separated $2$-page drawings with exactly $Z(n,n)$ crossings.

\section*{Acknowledgements}
This work was developed over the course of  AIM SQuaRE Workshops ``Crossing numbers of complete and complete bipartite graphs'', which started in 2021.
Part of this work was published~\cite{Balogh2023conf} in proceedings of XII Latin-American Algorithms, Graphs and Optimization Symposium (LAGOS 2023).

The authors thank Carolina Medina and Sergey Norin
for fruitful discussions. 
This work used the computing resources at the Center for Computational Mathematics, University of Colorado Denver, including the Alderaan cluster, supported by the National Science Foundation award OAC-2019089.

The work of third and last authors was also partially supported by the MEYS of the Czech Republic under the INTER-EXCELLENCE II program (project No. LUAUS26294).

 \apptocmd{\thebibliography}{%
     \setlength{\itemsep}{0pt}%
     \setlength{\parsep}{0pt}%
 }{}{}

\bibliographystyle{plainurl}
\bibliography{references.bib}

@article{Gethner2016,
  doi = {10.1002/jgt.22041},
  URLX = {https://doi.org/10.1002/jgt.22041},
  year = {2016},
  nomonth = mar,
  publisher = {Wiley},
  volume = {84},
  number = {4},
  pages = {552--565},
  author = {Ellen Gethner and Leslie Hogben and Bernard Lidick{\'{y}} and Florian Pfender and Amanda Ruiz and Michael Young},
  title = {On Crossing Numbers of Complete Tripartite and Balanced Complete Multipartite Graphs},
  journal = {J. Graph Theory}
}

@incollection {Decline,
    AUTHOR = {Guy, Richard K.},
     TITLE = {The decline and fall of {Z}arankiewicz's theorem},
 BOOKTITLE = {Proof {T}echniques in {G}raph {T}heory ({P}roc. {S}econd {A}nn
              {A}rbor {G}raph {T}heory {C}onf., {A}nn {A}rbor, {M}ich.,
              1968)},
     PAGES = {63--69},
 PUBLISHER = {Academic Press, New York},
      YEAR = {1969},
}

@article{Woodall93,
	Author = {Woodall, D. R.},
	Coden = {JGTHDO},
	Doi = {10.1002/jgt.3190170602},
	Fjournal = {Journal of Graph Theory},
	Issn = {0364-9024},
	Journal = {J. Graph Theory},
	Number = {6},
	Pages = {657--671},
	Title = {Cyclic-order graphs and {Z}arankiewicz's crossing-number conjecture},
	Volume = {17},
	Year = {1993},
}

@article{BaloghHill2019,
  year = {2019},
  nomonth = jan,
  publisher = {Society for Industrial {\&} Applied Mathematics ({SIAM})},
  volume = {33},
  number = {3},
  pages = {1261--1276},
  author = {J{\'{o}}zsef Balogh and Bernard Lidick{\'{y}} and Gelasio Salazar},
  title = {Closing in on {H}ill{\textquotesingle}s Conjecture},
  journal = {{SIAM} Journal on Discrete Mathematics},
  doi={10.1137/17M115885}
}

@InProceedings{goaoc2018,
  author =	{Xavier Goaoc and Alfredo Hubard and R{\'e}mi de Joannis de Verclos and Jean-S{\'e}bastien Sereni and Jan Volec},
  title =	{{Limits of Order Types}},
  booktitle =	{31st International Symposium on Computational Geometry (SoCG 2015)},
  pages =	{300--314},
  series =	{Leibniz International Proceedings in Informatics (LIPIcs)},
  ISBN =	{978-3-939897-83-5},
  ISSN =	{1868-8969},
  year =	{2015},
  volume =	{34},
  editor =	{Lars Arge and J{\'a}nos Pach},
  publisher =	{Schloss Dagstuhl--Leibniz-Zentrum fuer Informatik},
  address =	{Dagstuhl, Germany},
  doi={10.48550/arXiv.1811.02236},
}

@article{Razborov2007,
  doi = {10.2178/jsl/1203350785},
  URLX = {https://doi.org/10.2178/jsl/1203350785},
  year = {2007},
  nomonth = dec,
  publisher = {Cambridge University Press ({CUP})},
  volume = {72},
  number = {4},
  pages = {1239--1282},
  author = {Alexander A. Razborov},
  title = {Flag algebras},
  journal = {J. Symb. Log.}
}

@misc{NorinZwols,
author= {Sergey Norin and Yori Zwols},
title = {Tur\'an's {B}rickyard problem and flag algebras},
note = {Presentation at the BIRS Workshop on geometric and topological graph theory (13w5091), presented by S. Norin},
year = {October 1, 2013}, 
url = {http://videos.birs.ca/2013/13w5091/201310011538-Norin.mp4},
}

@article {Kleitman70,
    AUTHOR = {Kleitman, Daniel J.},
     TITLE = {The crossing number of {$K_{5,n}$}},
   JOURNAL = {J. Combinatorial Theory},
  FJOURNAL = {Journal of Combinatorial Theory},
    VOLUME = {9},
      YEAR = {1970},
     PAGES = {315--323},
      ISSN = {0021-9800},
   MRCLASS = {05.40},
  MRNUMBER = {280403},
  DOI = {10.1016/S0021-9800(70)80087-4},
}

@article{Beineke2010,
  doi = {10.1007/s00283-009-9120-4},
  URLX = {https://doi.org/10.1007/s00283-009-9120-4},
  year = {2010},
  nomonth = jan,
  publisher = {Springer Science and Business Media {LLC}},
  volume = {32},
  number = {2},
  pages = {41--48},
  author = {Lowell Beineke and Robin Wilson},
  title = {The Early History of the {B}rick {F}actory {P}roblem},
  journal = {Math. Intelligencer}
}

@article {Zarankiewicz54,
    AUTHOR = {Zarankiewicz, K.},
     TITLE = {On a problem of {P}. {T}uran concerning graphs},
   JOURNAL = {Fund. Math.},
  FJOURNAL = {Polska Akademia Nauk. Fundamenta Mathematicae},
    VOLUME = {41},
      YEAR = {1954},
     PAGES = {137--145},
      ISSN = {0016-2736},
   MRCLASS = {56.0X},
  MRNUMBER = {63641},
MRREVIEWER = {G. A. Dirac},
       DOI = {10.4064/fm-41-1-137-145},
       URLX = {https://doi.org/10.4064/fm-41-1-137-145},
}

@article {deKlerk06,
    AUTHOR = {de Klerk, E. and Maharry, J. and Pasechnik, D. V. and Richter,
              R. B. and Salazar, G.},
     TITLE = {Improved bounds for the crossing numbers of {$K_{m,n}$} and
              {$K_n$}},
   JOURNAL = {SIAM J. Discrete Math.},
  FJOURNAL = {SIAM Journal on Discrete Mathematics},
    VOLUME = {20},
      YEAR = {2006},
    NUMBER = {1},
     PAGES = {189--202},
      ISSN = {0895-4801},
   MRCLASS = {05C10 (68R10 90C22 90C25)},
  MRNUMBER = {2257255},
MRREVIEWER = {G\'{e}za T\'{o}th},
       DOI = {10.1137/S0895480104442741},
       URLX = {https://doi.org/10.1137/S0895480104442741},
}

@article {deKlerk07,
    AUTHOR = {de Klerk, Etienne and Pasechnik, Dmitrii V. and Schrijver,
              Alexander},
     TITLE = {Reduction of symmetric semidefinite programs using the regular
              {$\ast$}-representation},
   JOURNAL = {Math. Program.},
  FJOURNAL = {Mathematical Programming. A Publication of the Mathematical
              Programming Society},
    VOLUME = {109},
      YEAR = {2007},
    NUMBER = {2-3, Ser. B},
     PAGES = {613--624},
      ISSN = {0025-5610},
   MRCLASS = {90C22 (90C35)},
  MRNUMBER = {2296566},
MRREVIEWER = {Zheng Hai Huang},
       DOI = {10.1007/s10107-006-0039-7},
       URLX = {https://doi.org/10.1007/s10107-006-0039-7},
}

@manual{sagemath,
  Key          = {SageMath},
  Author       = {{The Sage Developers}},
  Title        = {{S}ageMath, the {S}age {M}athematics {S}oftware {S}ystem ({V}ersion 9.8)},
  URL         = { https://www.sagemath.org},
  Year         = {2023},
}

@article{Borchers1999,
  doi = {10.1080/10556789908805765},
  URLX = {https://doi.org/10.1080/10556789908805765},
  year = {1999},
  nomonth = jan,
  publisher = {Informa {UK} Limited},
  volume = {11},
  number = {1-4},
  pages = {613--623},
  author = {Brian Borchers},
  title = {{CSDP},  A {C} library for semidefinite programming},
  journal = {Optimization Methods and Software}
}

@misc{Aichholzer06,
  author = {Oswin Aichholzer},
  title = {Enumerating Order Types for Small Point Sets with Applications},
  year = 2006,
  url = {http://www.ist.tugraz.at/staff/aichholzer/research/rp/triangulations/ordertypes/},
  urldate = {2023-03-29},
}

@misc{Zwols2014,
  author = {Sergey Norine and Yori Zwols
},
  title = {Tur\'an's Brickyard Problem code},
  year = {2014},
  publisher = {GitHub},
  journal = {GitHub repository},
  url = {https://github.com/yozw/turan},
}

@manual{Mosek,
   author = "{MOSEK ApS}",
   title = "The MOSEK command line tools’s documentation. Version 10.0.",
   year = 2023,
   url = "https://docs.mosek.com/10.0/cmdtools/index.html"
 }

@article{Balogh2014,
  doi = {10.1017/s0963548314000820},
  URLX = {https://doi.org/10.1017/s0963548314000820},
  year = {2014},
  nomonth = dec,
  publisher = {Cambridge University Press ({CUP})},
  volume = {24},
  number = {4},
  pages = {658--679},
  author = {J{\'{o}}zsef Balogh and Ping Hu and Bernard Lidick{\'{y}} and Oleg Pikhurko and Bal{\'{a}}zs Udvari and Jan Volec},
  title = {Minimum Number of Monotone Subsequences of Length 4 in Permutations},
  journal = {Combin. Probab. Comput.}
}

@phdthesis{Baber2011,
  author = {Rahil Baber},
  title = {Some results in extremal combinatorics},
  year = {2011},
  school= {Department of Mathematics, University College London},
  url = {https://discovery.ucl.ac.uk/id/eprint/1306175/1/1306175.pdf},
}

@misc{Brosch2026,
  author = {Daniel Brosch and Oleg Pikhurko
},
  title = {Private communication, manuscript in preparation},
  year = {2026},
}

@misc{lpsdp2026,
  author = {Bernard Lidick\'y and Jan Volec
},
  title = {Solving Large {SDP}s from Flag Algebra Problems},
  year = {2026},
  note = {In preparation},
}

@article{Balogh2017RT,
  doi = {10.1016/j.jctb.2017.04.002},
  urlx = {https://doi.org/10.1016/j.jctb.2017.04.002},
  year = {2017},
  publisher = {Elsevier {BV}},
  volume = {126},
  pages = {83--113},
  author = {J{\'{o}}zsef Balogh and Ping Hu and Bernard Lidick{\'{y}} and Florian Pfender and Jan Volec and Michael Young},
  title = {Rainbow triangles in three-colored graphs},
  journal = {J. Combin. Theory Ser. B}
}

@article{Hladk2016,
  doi = {10.1007/s00493-015-2662-5},
  urlx = {https://doi.org/10.1007/s00493-015-2662-5},
  year = {2016},
  nomonth = jun,
  publisher = {Springer Science and Business Media {LLC}},
  volume = {37},
  number = {1},
  pages = {49--76},
  author = {Jan Hladk{\'{y}} and Daniel Kr{\'{a}}l' and Sergey Norin},
  title = {Counting flags in triangle-free digraphs},
  journal = {Combinatorica}
}

@article{Choi2020,
  doi = {10.1016/j.disc.2020.112015},
  urlx = {https://doi.org/10.1016/j.disc.2020.112015},
  year = {2020},
  nomonth = oct,
  publisher = {Elsevier {BV}},
  volume = {343},
  number = {10},
  pages = {112015},
  author = {Ilkyoo Choi and Bernard Lidick{\'{y}} and Florian Pfender},
  title = {Inducibility of directed paths},
  journal = {Discrete Mathematics}
}

@article{Balogh2022,
  doi = {10.1112/jlms.12568},
  urlx = {https://doi.org/10.1112/jlms.12568},
  year = {2022},
  nomonth = mar,
  publisher = {Wiley},
  volume = {106},
  number = {1},
  pages = {60--84},
  author = {J{\'{o}}zsef Balogh and Felix Christian Clemen and Bernard Lidick{\'{y}}},
  title = {Solving {T}ur{\'{a}}n{\textquotesingle}s tetrahedron problem for the $\ell_2$-norm},
  journal = {J. Lond. Math. Soc.}
}

@article{Brosch2022,
  title = {New lower bounds on crossing numbers of {$K_{m, n}$} from semidefinite programming},
  volume = {207},
  ISSN = {1436-4646},
  nourl = {http://dx.doi.org/10.1007/s10107-023-02028-1},
  DOI = {10.1007/s10107-023-02028-1},
  number = {1-2},
  journal = {Math. Program.},
  publisher = {Springer Science and Business Media LLC},
  author = {Brosch,  Daniel and C. Polak,  Sven},
  year = {2023},
  nomonth = Nov,
  pages = {693–715}
}

@article{Harary1963,
  title = {On the Number of Crossings in a Complete Graph},
  volume = {13},
  ISSN = {1464-3839},
  nourl = {http://dx.doi.org/10.1017/s0013091500025645},
  DOI = {10.1017/s0013091500025645},
  number = {4},
  journal = {Proc. Edinb. Math. Soc.},
  publisher = {Cambridge University Press (CUP)},
  author = {Harary,  Frank and Hill,  Anthony},
  year = {1963},
  nomonth = Dec,
  pages = {333–338}
}

@article{relizable,
author = {Ábrego, Bernardo and Aichholzer, Oswin and Fernández-Merchant, Silvia and Hackl, Thomas and Pammer, Jürgen and Pilz, Alexander and Ramos, Pedro and Salazar, Gelasio and Vogtenhuber, Birgit},
year = {2015},
pages = {57--60},
title = {All Good Drawings of Small Complete Graphs},
Journal = {In Proc. 31st European Workshop on Computational Geometry EuroCG ’15}
}

@article{Kiem2026,
  title = {The four-color {R}amsey multiplicity of triangles},
  volume = {179},
  ISSN = {0095-8956},
  nourl = {http://dx.doi.org/10.1016/j.jctb.2026.02.001},
  DOI = {10.1016/j.jctb.2026.02.001},
  journal = {J. Combin. Theory Ser. B},
  publisher = {Elsevier BV},
  author = {Kiem,  Aldo and Pokutta,  Sebastian and Spiegel,  Christoph},
  year = {2026},
  nomonth = July,
  pages = {19–70}
}

@misc{OAI:The-crossing-number-of-complete-bipartite-graphs-September-23-2026,
  author = {{OpenAI}},
  title = {{The crossing number of complete bipartite graphs}},
  howpublished = {OpenAI Math Release preprint
                  \href{https://github.com/openai/math/blob/main/preprints/The-crossing-number-of-complete-bipartite-graphs-September-23-2026/paper.pdf}{OAI:The-crossing-number-of-complete-bipartite-graphs-September-23-2026}},
  year = {2026}
}

@misc{OAI:The-crossing-number-of-complete-graphs-September-23-2026,
  author = {{OpenAI}},
  title = {{The crossing number of complete graphs}},
  howpublished = {OpenAI Math Release preprint
                  \href{https://github.com/openai/math/blob/main/preprints/The-crossing-number-of-complete-graphs-September-23-2026/paper.pdf}{OAI:The-crossing-number-of-complete-graphs-September-23-2026}},
  year = {2026}
}

@article{Balogh2023conf,
  title = {Crossing numbers of complete bipartite graphs},
  volume = {223},
  ISSN = {1877-0509},
  nourl = {http://dx.doi.org/10.1016/j.procs.2023.08.216},
  DOI = {10.1016/j.procs.2023.08.216},
  journal = {Procedia Computer Science},
  publisher = {Elsevier BV},
  author = {Balogh,  József and Lidický,  Bernard and Norin,  Sergey and Pfender,  Florian and Salazar,  Gelasio and Spiro,  Sam},
  year = {2023},
  pages = {78–87}
}

\end{document}